\documentclass{amsart}
\usepackage{amsmath}
\usepackage{amssymb}
\usepackage{amsthm}
\usepackage{enumerate}
\usepackage[dvips]{graphicx}
\usepackage{float, pict2e, bxeepic}

\usepackage[usenames,dvipsnames]{xcolor}

\newtheorem{theorem}{Theorem}
\newtheorem{prop}{Proposition}
\newtheorem{lemma}{Lemma}
\newtheorem{rem}{Remark}
\newtheorem{exmp}{Example}

\begin{document}
\title[Ortho-Grassmann graphs]
{Automorphisms and some geodesic properties of ortho-Grassmann graphs}
\author{Mark Pankov, Krzysztof Petelczyc, Mariusz \.Zynel}

\address{Mark Pankov: Faculty of Mathematics and Computer Science, 
University of Warmia and Mazury, S{\l}oneczna 54, 10-710 Olsztyn, Poland}
\email{pankov@matman.uwm.edu.pl}
\address{Krzysztof Petelczyc, Mariusz \.Zynel: Faculty of Mathematics, University of Bia{\l}ystok, Cio{\l}kowskiego 1 M, 15-245 Bia{\l}ystok, Poland}
\email{kryzpet@math.uwb.edu.pl, mariusz@math.uwb.edu.pl}

\maketitle

\begin{abstract}
Let $H$ be a complex Hilbert space. 
Consider the ortho-Grassmann graph $\Gamma^{\perp}_{k}(H)$
whose vertices are $k$-dimensional subspaces of $H$ (projections of rank $k$)
and two subspaces are connected by an edge in this graph if they are compatible and adjacent 
(the corresponding rank-$k$ projections commute and their difference is an operator of rank $2$).
Our main result is the following:
if $\dim H\ne 2k$, then every automorphism of $\Gamma^{\perp}_{k}(H)$ is
induced by a unitary or anti-unitary operator;
if $\dim H=2k\ge 6$, then every automorphism of $\Gamma^{\perp}_{k}(H)$
is induced by a unitary or anti-unitary operator
or it is the composition of such an automorphism and the orthocomplementary map.
For the case when $\dim H=2k=4$ the statement fails. 
To prove this statement we compare geodesics of length two in  ortho-Grassmann graphs
and characterise compatibility (commutativity) in terms of geodesics in 
Grassmann and ortho-Grassmann graphs.
At the end, we extend this result on generalised ortho-Grassmann graphs 
associated to conjugacy classes of finite-rank self-adjoint operators. 
\end{abstract}

\section{Introduction}

Classic Chow's  theorem \cite{Chow} states that all automorphisms of Grassmann graphs 
can be obtained from semi-linear automorphisms of 
the associated vector spaces or semi-linear isomorphisms to the dual vector spaces.
Recall that two vertices in the Grassmann graph formed by $k$-dimensional 
subspaces are connected by an edge if the corresponding subspaces are {\it adjacent}, i.e. 
their intersection is  $(k-1)$-dimensional.
There are results in spirit of Chow's theorem concerning automorphisms of polar Grassmann graphs 
\cite{Die,Pank-Grass} and analogues for matrix spaces \cite{Wan}.
Methods used to prove Chow-type theorems are based 
on a description of maximal cliques and their intersections in the corresponding graph. 

Recently, Chow's  theorem was successfully exploited in the discipline known as {\it preserver problems related to quantum mechanics} which provides various characterisations of quantum symmetries, i.e. 
unitary and anti-unitary operators \cite{GS,Geher,Pankov-book}.
The Grassmannian formed by $k$-dimensional subspaces of a complex Hilbert space $H$
can be identified with the conjugacy class of rank-$k$ projections and two such projections
are connected by an edge in the associated Grassmann graph if their difference is an operator of rank $2$
(i.e. the smallest possible).
The automorphisms of such Grassmann graphs cannot be characterised as the transformations induced by
unitary and anti-unitary operators and we look for a modification of this graph which will satisfy
that condition. 
Our choice is so-called {\it ortho-Grassmann graph}, where two subspaces are connected by an edge
if they are adjacent and compatible.
Then the corresponding projections commute and, as in the Grassmann graph,
their difference is an operator of rank $2$.

Recall that projections can be characterised as self-adjoint idempotents in the algebra of bounded operators. 
Observables in quantum mechanics are identified with (not necessarily bounded) self-adjoint operators  and two bounded observables are simultaneously observable if and only if the corresponding bounded self-adjoint operators commute (see, for example, \cite[Theorem 4.11]{Var}).
So, our ortho-adjacency relation is naturally interpreted in terms of quantum mechanics.

The main result (Theorem \ref{theorem-ortho-gr}) states that 
every automorphism of the ortho-Grass\-mann graph can be obtained from a unitary or anti-unitary operator
except the case when the graph consists of $2$-dimensional subspaces and $\dim H=4$.
A simple example shows that the statement fails for this case (Example \ref{exmp-4-2}).
In \cite{H} this statement is proved for the graph formed by $2$-dimensional subspaces.
There is an analogue of  our result for the set of all non-isotropic $k$-dimensional subspaces of a sesquilinear form under the assumption that the dimension of the associated vector space is not equal to $2k$ \cite{PZ}.
The method exploited in \cite{PZ} (a description of maximal cliques and their intersections)
does not work for the case when the graph is formed by subspaces whose dimension is the half of 
$\dim H$ (Section 3.3) and thus we  use completely different reasonings. 
Our proof is based on the comparison  of geodesics of length two in  ortho-Grassmann graphs
(Section 4) and a characterisation of compatibility (commutativity) in terms of geodesics in 
Grassmann and ortho-Grassmann graphs (Theorem \ref{theorem-comp}).

In quantum mechanics, rank-one projections, and consequently, $1$-dimensional subspaces
are identified with pure states \cite[Theorem 4.23]{Var}. 
By classic Wigner's theorem, all symmetries of pure states are induced by unitary or anti-unitary operators.
Observe that two $1$-dimensional subspaces are ortho-adjacent if and only if they are orthogonal,
i.e. for the Grassmannian of $1$-dimensional subspaces the ortho-adjacency and orthogonality relations
coincide; in this case, our result is  Uhlhorn's version of Wigner's theorem \cite{Uhlhorn}.
In the general case, the ortho-adjacency relation can be characterised in terms of principal angles as follows:
two subspaces are ortho-adjacent if and only if the principal angles between them are
$0,\dots,0,\pi/2$ (see, for example, \cite[Section 4.4]{Pankov-book}).
Transformations (not necessarily bijective) of Hilbert Grassmannians preserving the principal angles between any pair of subspaces are described in \cite{Molnar1, Molnar2}. 
Our result provides a description of bijections which preserve principal angles of type $0,\dots,0,\pi/2$.

There is an analogue of Chow's theorem 
for conjugacy classes of  finite-rank self-adjoint operators \cite{PPZ}. 
Using this statement and results from Section 4 
we extend the main result on such classes of operators (Theorem \ref{theorem-ortho-gr2}).

\section{Main result}
Let $V$ be a vector space over a field $F$.
For every integer $k$ satisfying $0<k<\dim H$ (if $H$ is infinite-dimensional, then for every integer $k>0$)
we denote by ${\mathcal G}_{k}(V)$ the Grassmannian formed by $k$-dimensional subspaces of $V$.
Two $k$-dimensional subspaces of  $V$ are called {\it adjacent} if their intersection is $(k-1)$-dimensional,
or equivalently, their sum is $(k+1)$-dimensional. 
In the case when $k=1$ or $\dim V=k+1$, any two distinct elements of ${\mathcal G}_{k}(V)$ are adjacent.
The {\it Grassmann graph} $\Gamma_k(V)$ is the simple graph whose vertex set is ${\mathcal G}_{k}(V)$ 
and two $k$-dimensional subspaces of $V$ are connected in this graph by an edge if they are adjacent.

Recall that a {\it semi-linear automorphism} of $V$ is a bijection $L:V\to V$ such that
$$L(x+y)=L(x)+L(y)$$
for all $x,y\in V$ and there is an automorphism $\sigma$ of the field $F$ satisfying
$$L(ax)=\sigma(a)L(x)$$
for all $a\in F$ and $x\in V$. 
Every semi-linear automorphism of $V$ induces an automorphism of $\Gamma_k(V)$.
If $\dim V=2k$ and $L:V\to V^{*}$ is a semi-linear isomorphism, then 
the map sending every $k$-dimensional subspace $X\subset V$ to $L(X)^0$ is an automorphism of $\Gamma_k(V)$
(for a subspace $Y$ of the dual vector space $V^{*}$ the annihilator $Y^{0}$ 
is formed by all $x\in V$ satisfying $y^{*}x=0$ for all $y^{*}\in Y$ and the dimension 
of $Y^{0}$ is equal to the codimension of $Y$).

\begin{theorem}[Chow \cite{Chow}]\label{theorem-chow}
Suppose that $k$ is an integer satisfying $1<k<\dim V- 1$ 
{\rm(}if $V$ is infinite-dimensional, then $k$ is an arbitrary integer greater than $1${\rm)}.
Then every automorphism of $\Gamma_k(V)$ is induced by a semi-linear automorphism 
of $V$ or a semi-linear isomorphism of $V$ to $V^{*}$ and 
the second possibility is realised only in the case when $\dim V=2k$.
\end{theorem}

\begin{rem}{\rm
In \cite{Chow} this statement was obtained for finite-dimensional vector spaces only;
but it holds for the general case; see, for example, \cite[Section 2.4]{Pankov-book}.
}\end{rem}

Let $H$ be a complex Hilbert space of dimension not less than $3$.
Closed subspaces of $H$ can be naturally identified with projections,
i.e. self-adjoint idempotents in the algebra of all bounded operators. 
A closed subspace $X\subset H$ corresponds to the projection $P_X$ whose image is $X$. 
Then ${\mathcal G}_k(H)$ coincides with the conjugacy class of rank-$k$ projections ${\mathcal P}_k(H)$.
Two $k$-dimensional subspaces $X,Y\subset H$  are adjacent if and only if 
the difference of the corresponding projections $P_X - P_Y$ is of rank $2$, i.e. of minimal rank
(the difference of two self-adjoint operators from the same conjugacy class cannot be of rank $1$).
Two closed subspaces $X,Y\subset H$ are called {\it compatible} if there is an orthonormal basis of $H$
such that $X$ and $Y$ are spanned by subsets of this basis; this is equivalent to the fact that  
the projections $P_X$ and $P_Y$ commute.
We say that $k$-dimensional subspaces $X,Y\subset H$ are {\it ortho-adjacent} 
if they are adjacent and compatible, i.e. 
\begin{equation}\label{eq-prad}
{\rm rank}(P_X - P_Y)=2\;\mbox{ and }\;P_XP_Y=P_YP_X.
\end{equation}
Denote by $\Gamma^{\perp}_k(H)$ the simple graph whose vertex set is ${\mathcal G}_k(H)$
and two $k$-dimensional subspaces of $H$ are connected by an edge in this graph if they are ortho-adjacent.
Then $\Gamma^{\perp}_k(H)$
can be considered as the graph whose vertex set is ${\mathcal P}_k(H)$ and 
two rank-$k$ projections $P_X,P_Y$ are connected by an edge if the conditions \eqref{eq-prad} are satisfied.
In contrast to the Grassmann graph $\Gamma_k(H)$, for $k=1$ or $\dim H=k+1$
the graph $\Gamma^{\perp}_k(H)$ contains pairs of distinct vertices which are not connected by an edge.
In the next section, we show that this graph is connected.

An invertible linear operator on $H$ is {\it unitary} if it preserves the inner product.
An invertible conjugate-linear operator $U$ on $H$
(i.e. satisfying $U(x+y)=U(x)+U(y)$ and $U(ax)=\overline{a}U(x)$ for all $x,y\in H$ and all scalars $a$)
is called {\it anti-unitary} if
$$\langle U(x),U(y)\rangle=\overline{\langle x,y \rangle}$$
for all $x,y\in H$.
Unitary and anti-unitary operators induce automorphisms of the graph $\Gamma^{\perp}_k(H)$.
If $\dim H=n$ is finite, then the orthocomplementary map 
sending every $k$-dimensional subspace $X\subset H$ to the orthogonal complement $X^{\perp}$
is an isomorphism between $\Gamma^{\perp}_{k}(H)$ and $\Gamma^{\perp}_{n-k}(H)$.

Our main result is the following.

\begin{theorem}\label{theorem-ortho-gr}
If $\dim H\ne 2k$ {\rm(}in particular, if $H$ is infinite-dimensional{\rm)},
then every automorphism of the graph $\Gamma^{\perp}_{k}(H)$ is
induced by a unitary or anti-unitary operator. 
In the case when $\dim H=2k\ge 6$, every automorphism of $\Gamma^{\perp}_{k}(H)$
is induced by a unitary or anti-unitary operator
or it is the composition of such an automorphism and the orthocomplementary map.
\end{theorem}

\begin{rem}\label{rem-theorem2}{\rm
If $X$ is a $k$-dimensional subspace of $H$, then $P_{X^{\perp}}={\rm Id}-P_{X}$ and
$$P_{U(X)}=UP_{X}U^{*}$$
for every unitary or anti-unitary operator $U$ on $H$.
Consider $\Gamma^{\perp}_{k}(H)$ as the graph whose vertices are rank-$k$ projections.
Then Theorem \ref{theorem-ortho-gr} can be reformulated as follows.
If $\dim H\ne 2k$ and $f$ is an automorphism of $\Gamma^{\perp}_{k}(H)$, then there is 
a unitary or anti-unitary operator $U$ on $H$ such that
\begin{equation}\label{eq-unitary}
f(P)=UPU^{*}
\end{equation}
for all $P\in {\mathcal P}_{k}(H)$.
If $\dim H=2k\ge 6$ and $f$ is an automorphism of $\Gamma^{\perp}_{k}(H)$, then there is 
a unitary or anti-unitary operator $U$ on $H$ such that \eqref{eq-unitary}
holds for all $P\in {\mathcal P}_{k}(H)$ or we have
$$f(P)=U({\rm Id}-P)U^{*}$$
for every $P\in {\mathcal P}_{k}(H)$.
}\end{rem}

The above statement fails if $\dim H=2k=4$, see Example \ref{exmp-4-2}.

Two $1$-dimensional subspaces of $H$ are ortho-adjacent if and only if they are orthogonal. 
So, Theorem \ref{theorem-ortho-gr} coincides with 
Uhlhorn's version of Wigner's theorem concerning bijective transformations of ${\mathcal G}_{1}(H)$
preserving the orthogonality relation in both directions \cite{Uhlhorn}.
Therefore, it is sufficient to prove Theorem \ref{theorem-ortho-gr} only for the case when 
$4\le 2k\le \dim H$.
Indeed, if $\dim H=n$ is finite and $2k>n$, then
for every automorphism $f$ of $\Gamma^{\perp}_{k}(H)$ 
the map sending every $(n-k)$-dimensional subspace $X\subset H$
to $(f(X^{\perp}))^{\perp}$ is an automorphism of $\Gamma^{\perp}_{n-k}(H)$;
if this automorphism is induced by a unitary or anti-unitary operator,
then $f$ is induced by the same operator.

Our proof of  Theorem \ref{theorem-ortho-gr}  is based on
a characterisation  of adjacency in terms of ortho-adjacency (Section 4)
which  immediately implies that every automorphism $f$ of $\Gamma^{\perp}_{k}(H)$ is an automorphism of 
$\Gamma_{k}(H)$ (except the case when $\dim H=2k=4$).
It follows from Theorem \ref{theorem-chow} that  there is a semi-linear automorphism $L:H\to H$
such that one of the following possibilities is realised: 
\begin{enumerate}
\item[$\bullet$] $f(X)=L(X)$ for every $k$-dimensional subspace $X\subset H$,
\item[$\bullet$] $\dim H=2k$ and $f(X)=L(X)^{\perp}$ for every $k$-dimensional subspace $X\subset H$.
\end{enumerate}
We cannot immediately assert that $L$ is a scalar multiple of a unitary or anti-unitary operator.
This follows from our characterisation of compatibility in terms of geodesics in 
the graphs  $\Gamma^{\perp}_{k}(H)$ and $\Gamma_{k}(H)$
(Theorem \ref{theorem-comp}).

\begin{exmp}\label{exmp-4-2}{\rm
Suppose that $\dim H=4$. 
Then for any ortho-adjacent $2$-dimensional subspaces $X,Y\subset H$
the subspaces $X,Y^{\perp}$ are also ortho-adjacent
(we take any orthonormal basis of $H$ whose subsets span $X,Y$ and observe that
$Y^{\perp}$ also is spanned by a subset of this basis).
Let us fix a certain $2$-dimensional subspace $X\subset H$
and consider the transformation of ${\mathcal G}_{2}(H)$ which 
transposes  $X,X^{\perp}$ and leaves fixed all other $2$-dimensional subspaces of $H$.
This is an automorphism of $\Gamma^{\perp}_{2}(H)$.
More generally, we can take any subset ${\mathcal X}\subset {\mathcal G}_{2}(H)$
such that $X^{\perp}\in {\mathcal X}$ for all $X\in {\mathcal X}$
and the transformation of ${\mathcal G}_{2}(H)$ which sends all $X\in {\mathcal X}$
to $X^{\perp}$ and leaves all elements of ${\mathcal G}_{2}(H)\setminus {\mathcal X}$ fixed.
As above, we obtain an automorphism of $\Gamma^{\perp}_{2}(H)$.
Our conjecture is the following: if $\dim H=4$ and $f$ is an automorphism of $\Gamma^{\perp}_{2}(H)$,
then there is a unitary or anti-unitary operator $U$ such that for every $2$-dimensional subspace $X\subset H$
we have $f(X)=U(X)$ or $f(X)=U(X)^{\perp}$.
}\end{exmp}

\section{Grassmann graphs and ortho-Grassmann graphs}

\subsection{Some elementary properties}
Recall that the {\it distance} $d(v,w)$ between two vertices $v,w$ in a connected simple graph 
is defined as the smallest number $m$ such that there is a path of length $m$ 
(i.e. a path which consists of $m$ edges) 
connecting these vertices.
A path connecting  $v$ and $w$ is called a {\it geodesic} if its length is $d(v,w)$. 

We will use some basic properties of the distance in Grassmann graphs 
(see, for example, \cite[Section 2.3]{Pankov-book}).
The Grassmann graph $\Gamma_{k}(H)$ is connected;
furthermore, for any $k$-dimensional subspaces $X,Y\subset H$ the following assertions are fulfilled:
\begin{enumerate}
\item[(D1)] the distance between $X$ and $Y$ in $\Gamma_k(H)$
is equal to $k-\dim(X\cap Y)$;
\item[(D2)] every $k$-dimensional subspace in a geodesic of $\Gamma_k(H)$ connecting $X$ with $Y$
contains $X\cap Y$ and is contained in $X+Y$;
\item[(D3)] if the distance between $X$ and $Y$ is equal to $m$
and $Z$ is a $k$-dimensional subspace of $H$ such that the distances from $Z$ to $X$ and $Y$ are equal to
$i$ and $m-i$ (respectively), then there is a geodesic in $\Gamma_k(H)$ connecting $X$ with $Y$
and containing $Z$.
\end{enumerate}
If $X,Y,Z$ are mutually adjacent $k$-dimensional subspaces of $H$, 
then at least one of the following possibilities is realised: 
\begin{enumerate}
\item[(T1)] $X\cap Y\cap Z$ is $(k-1)$-dimensional,
\item[(T2)] $X+Y+Z$ is $(k+1)$-dimensional.
\end{enumerate}
Note that there are triples of $k$-dimensional subspaces $X,Y,Z$ satisfying both (T1) and (T2).
In the case when (T1) fails and (T2) holds, $X\cap Y\cap Z$ is $(k-2)$-dimensional and $k>1$.
If (T1) holds and (T2) fails, then $X+Y+Z$ is $(k+2)$-dimensional and $k<\dim H-1$.
In the case when $1<k<\dim H-1$, there are precisely two types of maximal cliques in $\Gamma_{k}(H)$:
\begin{enumerate}
\item[$\bullet$] the {\it star} ${\mathcal S}(S)$, $S\in {\mathcal G}_{k-1}(H)$
which consists of all $k$-dimensional subspaces containing $S$;
\item[$\bullet$] the {\it top} ${\mathcal G}_{k}(U)$, $U\in {\mathcal G}_{k+1}(H)$.
\end{enumerate}
Every automorphism of $\Gamma_{k}(H)$ preserves types of maximal cliques
(stars go to stars and tops go to tops) or it sends all stars to tops and all tops to stars.
The first possibility is realised if and only if this automorphism is induced by 
a semi-linear automorphism of $H$. See \cite[Section 2.4]{Pankov-book} for the details.

Now, we describe all possible connections in $\Gamma^{\perp}_k(H)$ between
two non-compatible adjacent $k$-dimension subspaces.

\begin{lemma}\label{lemma-ad-1}
Let $X,Y$ be adjacent $k$-dimensional subspaces of $H$ and $k<\dim H-1$.
A $k$-dimensional subspace $Z\not\subset X+Y$ is ortho-adjacent to both $X,Y$ if and only if 
$Z = (X\cap Y) + P$ for some $1$-dimensional subspace $P$ orhogonal to $X+Y$.
\end{lemma}

\begin{proof}
The condition $k<\dim H-1$ implies that $X+Y$ is a proper subspace of $H$.
For every $1$-dimensional subspace $P\subset (X+Y)^{\perp}$
the $k$-dimensional subspace $P+(X\cap Y)$ is ortho-adjacent to both $X,Y$.
Conversely, if $Z$ is a $k$-dimensional subspace ortho-adjacent to both $X,Y$ and $Z\not\subset X+Y$,
then $Z$ contains $X\cap Y$ as a hyperplane (the case (T1)) and the $1$-dimensional  orthogonal complement of $X\cap Y$ in $Z$
is orthogonal to both $X,Y$, and consequently, to $X+Y$.
\end{proof}

\begin{lemma}\label{lemma-ad-2}
Let $X,Y$ be adjacent $k$-dimensional subspaces of $H$ with $k>1$
and let $S$ be the $2$-dimensional orthogonal complement of $X\cap Y$ in $X+Y$.
A $k$-dimensional subspace $Z\subset X+Y$ is ortho-adjacent to both $X,Y$ if and only if 
$Z = S+W$ for some $(k-2)$-dimensional subspace $W\subset X\cap Y$.
\end{lemma}

\begin{proof}
For every $(k-2)$-dimension subspace $W\subset X\cap Y$ the $k$-dimensional subspace
$W+S$ is ortho-adjacent to both $X,Y$.
Indeed, $X,Y,W+S$ are hyperplanes in $X+Y$ and $W+S$ contains 
the $1$-dimensional subspaces 
\begin{equation}\label{eq-2}
X^{\perp}\cap(X+Y)\mbox{ and }Y^{\perp}\cap(X+Y)
\end{equation}
orthogonal to 
$X$ and $Y$, respectively.
If a $k$-dimensional subspace $Z\subset X+Y$ is ortho-adjacent to both $X,Y$,
then it contains the $1$-dimensional subspaces \eqref{eq-2}, and consequently, their sum $S$.
This implies that $Z$ intersects $X\cap Y$ in a $(k-2)$-dimensional subspace.
\end{proof}

\begin{prop}
The graph $\Gamma^{\perp}_k(H)$ is connected.
\end{prop}

\begin{proof}
By Lemmas \ref{lemma-ad-1} and \ref{lemma-ad-2},
for any two non-compatible adjacent $k$-dimensional subspaces of $H$
there is a $k$-dimensional subspace ortho-adjacent to each of them.
The required statement is a consequence of the fact that the Grassmann graph $\Gamma_k(H)$ is connected.
\end{proof}

The distance between two non-compatible adjacent $k$-dimensional subspaces of $H$
in $\Gamma^{\perp}_{k}(H)$ is equal to $2$. 
The following observation concerns the number of geodesics in $\Gamma^{\perp}_{k}(H)$
connecting such subspaces.

\begin{lemma}\label{lemma-adnoncom}
If $k$-dimensional subspaces $X,Y\subset H$ are adjacent and non-compatible,
then the number of geodesics in $\Gamma^{\perp}_{k}(H)$ connecting $X$ and $Y$ is infinite
except the case when $\dim H=2k=4$.
In this exceptional case, there are precisely $2$ such geodesics.
\end{lemma}

\begin{proof}
If $Z$ is a $k$-dimensional subspace ortho-adjacent to both $X,Y$, 
then one of the following possibilities is realised:
\begin{enumerate}
\item[(1)] $k<\dim H-1$ and $Z$ is the sum of $X\cap Y$ and a $1$-dimensional subspace orthogonal to $X+Y$,
\item[(2)] $k>1$ and $Z$ is the sum of a $(k-2)$-dimensional subspace of $X\cap Y$
and the $2$-dimensional orthogonal complement of $X\cap Y$ in $X+Y$.
\end{enumerate}
If $\dim H=2k=4$, then there is precisely one $Z$ satisfying (1) and 
precisely one satisfying (2). 
For the remaining cases there are infinitely many $Z$ satisfying at least one of the conditions
(1), (2).
\end{proof}

\subsection{A characterisation of compatibility in terms of geodesics}
Every geodesic of $\Gamma^{\perp}_k(H)$ is a geodesic in $\Gamma_k(H)$, but the converse fails. 
If $X,Y$ are compatible $k$-dimensional subspaces of $H$, 
then every geodesic of $\Gamma_k(H)$ connecting $X$ and $Y$ is a geodesic in $\Gamma^{\perp}_k(H)$; 
furthermore, any two $k$-dimensional subspaces in such a geodesic
are compatible; see \cite[Lemma 4.31]{Pankov-book}.

\begin{theorem}\label{theorem-comp}
Two $k$-dimensional subspaces $X,Y\subset H$ are compatible if and only if 
every geodesic of $\Gamma_k(H)$ connecting $X$ and $Y$ is a geodesic in $\Gamma^{\perp}_k(H)$.
\end{theorem}

\begin{proof}
By  \cite[Lemma 4.31]{Pankov-book}, we have to prove the following:
if every geodesic of $\Gamma_k(H)$ connecting $X$ and $Y$ is a geodesic in 
$\Gamma^{\perp}_k(H)$, then $X,Y$ are compatible.

We start from the case when $X\cap Y=0$, i.e.
the distance between $X$ and $Y$ in $\Gamma_k(H)$ is equal to $k$ (the property (D1)).
We need to show that $X$ and $Y$ are orthogonal.
The subspace $X+Y$ is $2k$-dimensional and we denote by $X'$ the $k$-dimensional subspace
which is the orthogonal complement of $X$ in $X+Y$.
Let us take any $1$-dimensional subspace $P\subset Y$ and a $(k-1)$-dimensional subspace $N\subset X$.
The $k$-dimensional subspace $P+N$ is adjacent to $X$ and the distance between 
$P+N$ and $Y$ in $\Gamma_k(H)$ is equal to $k-1$ (the property (D1)).
By the property (D3),
there is a geodesic of $\Gamma_k(H)$ which connects $X$ with $Y$ and contains $P+N$.
By our assumption, this is a geodesic in $\Gamma^{\perp}_k(H)$
and $P+N$ is ortho-adjacent to $X$.
Therefore, $P+N$ contains a $1$-dimensional subspace $Q$ orthogonal to $X$,
in other words, $P+N$ intersects $X'$ precisely in $Q$. 
Suppose that $Q$ is distinct from $P$.
We choose a $(k-1)$-dimensional subspace $M\subset X$
such that $P+M$ does not contain $Q$
(such subspace exists, since the intersection of all $P+N$, where $N$ is a $(k-1)$-dimensional subspace of $X$,
coincides with $P$ and $P\ne Q$ by assumption). 
Then $P+M$ intersects $X'$ in a $1$-dimensional subspace $Q'$ distinct from $Q$.
The $2$-dimensional subspace $Q'+Q$ is contained in the intersection of $X'$ and $P+X$.
Since $X$ is a hyperplane of $P+X$, the subspace $Q'+Q$ has a non-zero intersection with $X$
which is impossible. This contradiction means that $P=Q$.
So, we have established that every $1$-dimensional subspace $P\subset Y$ is contained in $X'$.
Therefore, $Y=X'$, and consequently, $X$ is orthogonal to $Y$.

Suppose that $\dim(X\cap Y)=m$ and $0<m<k-1$
(the case when $X$ and $Y$ are adjacent is trivial). 
Then $$\dim (X+Y)=2k-m$$
and the distance between $X$ and $Y$ in $\Gamma_{k}(H)$ is equal to $k-m$ (the property (D1)).
Denote by $V$ the orthogonal complement of $X\cap Y$ in $X+Y$.
The dimension of this subspace is equal to $2(k-m)$.
The intersection of the $(k-m)$-dimensional subspaces 
$$X'=X\cap V\;\mbox{ and }\;Y'=Y\cap V$$ 
is zero and the distance between them in $\Gamma_{k-m}(H)$ is equal to $k-m$.
If 
\begin{equation}\label{eq-geo1}
X'=Z'_{0},Z'_{1},\dots, Z'_{k-m}=Y'
\end{equation}
is a geodesic in $\Gamma_{k-m}(H)$,
then each $Z'_i$ is contained in $V$ (the property (D2)) and
\begin{equation}\label{eq-geo2}
X=Z_{0}, Z_{1},\dots,Z_{k-m}=Y\;\mbox{ with }\;Z_{i}=Z'_i +(X\cap Y), i\in \{0,1,\dots,k-m\}
\end{equation}
is a geodesic in $\Gamma_{k}(H)$. 
By our assumption, \eqref{eq-geo2} is a geodesic in $\Gamma^{\perp}_{k}(H)$
which implies that $Z_{i-1}$ and $Z_i$ are ortho-adjacent for every $i\in \{1,\dots,k-m\}$.
The latter guarantees that $Z'_{i-1}$ and $Z'_i$ are ortho-adjacent for every $i\in \{1,\dots,k-m\}$, i.e.
\eqref{eq-geo1} is a geodesic in $\Gamma^{\perp}_{k-m}(H)$.
So, every geodesic of $\Gamma_{k-m}(H)$ connecting $X'$ and $Y'$
is a geodesic in $\Gamma^{\perp}_{k-m}(H)$.
Applying arguments from the previous paragraph, we establish that $X'$ and $Y'$ are orthogonal. 
This means that $X$ and $Y$ are compatible.
\end{proof}

As an application of Theorem \ref{theorem-comp}, we prove the following statement 
which will be exploited to prove Theorem \ref{theorem-ortho-gr}.

\begin{lemma}\label{lemma-aut}
Suppose that $4\le 2k\le \dim H$ and 
$f$ is a bijective transformation of ${\mathcal G}_{k}(H)$ which is an automorphism of both 
$\Gamma_{k}(H)$ and $\Gamma^{\perp}_{k}(H)$.
Then $f$ is induced by a unitary or anti-unitary operator or it is the composition of 
a bijection induced by a unitary or anti-unitary operator and the orthocomplementary map.
The second possibility is realised only in the case when $\dim H=2k$.
\end{lemma}

\begin{proof}
If $f$ sends stars to stars and tops to tops, then $f$ is induced by a semi-linear automorphism of $H$;
otherwise, this property holds for the composition of $f$ and the orthocomplementary map
(since the orthocomplementation transfers stars to tops and tops to stars).
Therefore, there is a semilinear automorphism $U$ of $H$ such that 
one of the following possibilities is realised:
\begin{enumerate}
\item[$\bullet$] $f(X)=U(X)$ for all $X\in {\mathcal G}_{k}(H)$,
\item[$\bullet$] $\dim H=2k$ and $f(X)=U(X)^{\perp}$ for all $X\in {\mathcal G}_{k}(H)$.
\end{enumerate}
We need to show that $U$ is a scalar multiple of a unitary or anti-unitary operator
(since $U$ and $aU$ induce the same transformation of ${\mathcal G}_{k}(H)$ for every non-zero scalar $a$).
It is sufficient to consider the first case 
(in the second case, we replace $f$ by the composition of $f$ and the orthocomplementary). 
Since $f$ is an automorphism of both $\Gamma^{\perp}_{k}(H)$ and $\Gamma_{k}(H)$,
Theorem \ref{theorem-comp} guarantees that $f$ is compatibility preserving in both directions. 
Two $k$-dimensional subspaces of $H$ are orthogonal if and only if they are compatible 
and the distance between them in $\Gamma_{k}(H)$ is equal to $k$.
This means that $f$ is orthogonality preserving in both directions and
$U$ sends orthogonal vectors to orthogonal vectors
which implies that $U$ is a scalar multiple of a unitary or anti-unitary operator
\cite[Proposition 4.2]{Pankov-book}.
\end{proof}

\subsection{Remarks}
It was noted above that it is sufficient to prove Theorem \ref{theorem-ortho-gr} only 
for the case when $4\le 2k\le \dim H$. 
If $\dim H\ne 2k$, then Theorem \ref{theorem-ortho-gr} is a direct consequence of
Lemma \ref{lemma-aut} and the following statement.

\begin{lemma}[Lemma 4.37 in \cite{Pankov-book}]\label{lemma-old}
If $2<2k< \dim H$, then every ortho-adjacency preserving injective transformation of ${\mathcal G}_{k}(H)$
is adjacency preserving.
\end{lemma}

\begin{rem}{\rm
For $k=2$ Lemma \ref{lemma-old} is proved in \cite{H}.
}\end{rem}

In \cite{PZ}, an analogue of Theorem \ref{theorem-ortho-gr} is proved 
for the set of all non-isotropic $k$-dimensional subspaces of a sesquilinear form
under the assumption that the dimension of the associated vector space is distinct from $2k$.
Methods used to prove Lemma \ref{lemma-old} as well as methods exploited in \cite{PZ}
are based on some properties of maximal cliques in ortho-Grassmann graphs.

Every maximal clique of $\Gamma^{\perp}_{k}(H)$ is the intersection of 
a star or a top (a maximal clique of  $\Gamma_{k}(H)$) with the set of all $k$-dimensional subspaces spanned by subsets of an orthonormal basis of $H$
such that the $(k-1)$-dimensional or $(k+1)$-dimensional subspace associated with the star or the top
(respectively) is also spanned by a subset of this basis \cite[Section 4.6]{Pankov-book}.
Such a clique will be called an {\it ortho-star} or an {\it ortho-top}, respectively. 
Every ortho-top contains precisely $k+1$ elements; 
an ortho-star  contains precisely $n-k+1$ elements if $\dim H=n$ is finite, and it is infinite if 
$H$ is infinite-dimensional \cite[Lemma 4.30]{Pankov-book}. 
In this  way ortho-stars can be easily distinguished from ortho-tops, but not in the case when $\dim H = 2k$ and they have the same number of elements.

If $2<2k< \dim H$, 
then every ortho-adjacency preserving injective transformation of ${\mathcal G}_{k}(H)$
sends ortho-stars to subsets of ortho-stars. This observation is a crucial tool in the proof of Lemma \ref{lemma-old}.
Similarly, transformations considered in \cite{PZ} send ortho-stars to ortho-stars and ortho-tops to ortho-tops.

It was conjectured in \cite{PZ} that the method used to prove Chow's theorem 
can be modified for the case when $\dim H=2k$.
The key argument is the following:
if the intersection of a star and a top is non-empty, then the number of elements in this intersection
is greater than the number of elements in the intersection of two distinct stars or 
the intersection of two distinct tops. 
Consequently,  if one star goes to a top under an automorphism of the Grassmann graph, then all stars go to tops and all tops go to stars.

Now, we explain why the same arguments do not work for ortho-stars and ortho-tops
in the case when $\dim H=2k$.
The intersection of two distinct maximal cliques of $\Gamma^\perp_k(H)$ depends 
on the relations between the associated $(k-1)$-dimensional or $(k+1)$-dimensional subspaces 
and orthonormal bases. A direct verification shows that the following assertions are fulfilled:
\begin{enumerate}
\item[$\bullet$] The number of elements in the intersection of two distinct ortho-stars
associated to the same $(k-1)$-dimensional subspace can take any value $m\in\{0,1,\dots, k-1\}$
and the same holds for the intersection of two distinct ortho-tops corresponding to the same
$(k+1)$-dimensional subspace.
\item[$\bullet$] The intersection of two ortho-stars associated to distinct $(k-1)$-dimensional subspaces
is empty or consists of one element and the same holds  for 
the intersection of two ortho-tops corresponding to  distinct $(k+1)$-dimensional subspaces.
\item[$\bullet$] The number of elements in the intersection of an ortho-star and an ortho-top 
can take any value $m\in\{0,1,2\}$.
\end{enumerate}
Consider ortho-stars ${\mathcal S}_1,{\mathcal S}_2$ associated to distinct $(k-1)$-dimensional subspaces
and such that ${\mathcal S}_1\cap{\mathcal S}_2$ is one element.
If $\mathcal S_1$ goes to an ortho-top $\mathcal T$ under an automorphism of 
$\Gamma^\perp_k(H)$, then this automorphism sends $\mathcal S_2$ to an ortho-star
or to an ortho-top intersecting $\mathcal T$ precisely in one element. Therefore,
to claim that an automorphism of $\Gamma^\perp_k(H)$ either preserves or exchanges
two types of maximal cliques, we need different arguments.

Our proof of  Theorem \ref{theorem-ortho-gr} is based on 
a characterisation of adjacency  in terms of ortho-adjacency and 
do not depend on the relation between $\dim H$ and $k$.

\section{A characterization of adjacency in terms of ortho-adjacency and proof of Theorem \ref{theorem-ortho-gr}}
In this section, we investigate geodesics of length $2$ in the ortho-Grassmann graph $\Gamma^{\perp}_{k}(H)$.
If $\dim H=n$ is finite, then the orthocomplementary map is an isomorphism between
$\Gamma^{\perp}_{k}(H)$ and $\Gamma^{\perp}_{n-k}(H)$. 
So, we assume that $4\le 2k\le \dim H$ without loss of generality.

Let $X$ and $Y$ be distinct $k$-dimensional subspaces of $H$ at distance $2$ in $\Gamma^{\perp}_{k}(H)$.
Then $X,Y$ are adjacent and non-compatible (this case was considered in Subsection 3.1)
or the distance between $X$ and $Y$ in $\Gamma_{k}(H)$ is equal to $2$ and $X\cap Y$
is $(k-2)$-dimensional.
The following example concerns the second case. 

\begin{exmp}{\rm
Suppose that $X,Y$ are compatible $k$-dimensional subspaces of $H$ whose intersection is $(k-2)$-dimensional
and consider an orthonormal basis of $H$ such that $X$ and $Y$ are spanned by subsets of this basis.
There is a $k$-dimensional subspace spanned by  a subset of this basis and ortho-adjacent to both $X,Y$.
}\end{exmp}

In the case when $X\cap Y$ is $(k-2)$-dimensional,
the existence of a $k$-dimensional subspace ortho-adjacent to both $X,Y$
does not imply that $X$ and $Y$ are compatible. 
We give an example for $k=2$, but the general case is similar.

\begin{exmp}{\rm
Let us  take any $2$-dimensional subspace $Z\subset H$. 
Consider a $1$-dimensional subspace $P\subset Z$ and 
the $1$-dimensional subspace $Q=Z\cap P^{\perp}$ which is 
the unique $1$-dimensional subspace of $Z$ orthogonal to $P$. 
Let $X$ be a $2$-dimensional subspace containing $Q$ and orthogonal to $P$
(since $\dim H\ge 4$, such a subspace is not unique).
Similarly, we take a $1$-dimensional subspace $P'\subset Z$ distinct from $P,Q$,
the  $1$-dimensional subspace $Q'=Z\cap P'^{\perp}$ (which is also distinct from $P,Q$)
and any $2$-dimensional subspace $Y$ containing $Q'$ and orthogonal to $P'$.
We can choose $X,Y$ such that $X\cap Y=0$ and $X,Y$ are non-orthogonal.
Then $X,Y$ are non-compatible; on the other hand,
each of $X,Y$ is ortho-adjacent to $Z$.
}\end{exmp}

\begin{lemma}\label{lemma2-1}
Let $X$ and $Y$ be compatible $k$-dimensional subspaces of $H$ whose intersection is $(k-2)$-dimensional.
Then for every $k$-dimensional subspace $Z\subset H$ ortho-adjacent to both $X,Y$
there are precisely two $k$-dimensional subspaces of $H$ ortho-adjacent to each of $X,Y,Z$.
\end{lemma}

\begin{proof}
Since $X,Y,Z$ are mutually compatible, there is an orthonormal basis $B$ of $H$
such that each of these subspaces is spanned by a subset of $B$.
Observe that the $(k-2)$-dimensional subspace $X\cap Y$ and the $2$-dimensional subspaces 
$$X'=X\cap (X\cap Y)^{\perp}\;\mbox{ and }\; Y'=Y\cap (X\cap Y)^{\perp}$$
are also spanned by subsets of $B$.
Since $X,Z,Y$ is a geodesic in $\Gamma^{\perp}_{k}(H)$,
the subspace $Z$ contains $X\cap Y$ and intersects $X'$ and $Y'$ in $1$-dimensional subspaces $P$ and $Q$,
respectively.
Denote by $P'$ and $Q'$ the $1$-dimensional subspaces 
which are the  orthogonal complements of $P$ and $Q$ in $X'$ and $Y'$, respectively.
Each of the $1$-dimensional subspaces $P,P',Q,Q'$ contains a vector from $B$.
The $k$-dimensional subspaces
$$Z_1= P' +(X\cap Y)+Q\;\mbox{ and }\; Z_2=P+(X\cap Y)+Q'$$
are  ortho-adjacent to $X,Y,Z$ and spanned by subsets of $B$.

Suppose that $Z'$ is a $k$-dimensional subspace ortho-adjacent to $X,Y,Z$.
Then $X,Y,Z,Z'$ are mutually compatible and there is an orthonormal basis
$B'$ such that each of these subspaces is spanned by a subset of $B'$.
The subspaces $X\cap Y, X',Y'$ are  also spanned by vectors from $B'$
and each of the $1$-dimensional subspaces $P,P',Q,Q'$ contains a vector from $B'$.
As above, we establish that $Z'$ contains $X\cap Y$,
the subspaces $X'\cap Z'$, $Y'\cap Z'$ are $1$-dimensional and 
$$Z'= (X'\cap Z')+(X\cap Y)+(Y'\cap Z').$$
Since  $X'\cap Z'$ is a $1$-dimensional subspace of $X'$ containing a vector from $B'$,
it coincides with $P$ or $P'$.
Similarly, $Y'\cap Z'$ coincides with $Q$ or $Q'$.
Since $Z$ and $Z'$ are ortho-adjacent, we obtain that $Z'$ is $Z_1$ or $Z_2$.
\end{proof}

\begin{lemma}\label{lemma2-2}
Suppose that $X$ and $Y$ are $k$-dimensional subspaces of $H$ whose intersection is $(k-2)$-dimensional
and there are two ortho-adjacent $k$-dimensional subspaces $Z,Z'\subset H$ such that each of these subspaces 
is ortho-adjacent to both $X,Y$. Then $X$ and $Y$ are compatible.
\end{lemma}

\begin{proof}
First, we show that the $(k-1)$-dimensional subspace $Z\cap Z'$
is contained in $X$ or $Y$.
If $Z\cap Z'$ is not contained in $X$, then $Z$ and $Z'$ intersect $X$ in distinct 
$(k-1)$-dimensional subspaces whose sum coincides with $X$.
Similarly, $Z\cap Z' \not\subset Y$ implies that  $Z$ and $Z'$ intersect $Y$ in distinct 
$(k-1)$-dimensional subspaces whose sum is $Y$.
Then the $(k+1)$-dimensional subspace $Z+Z'$ contains both $X,Y$
and $\dim(X+Y)<k+2$
which contradicts the assumption that $X\cap Y$ is $(k-2)$-dimen\-sional.

Without loss of generality, we can assume that $Z\cap Z'$ is a $(k-1)$-dimensional subspace of $X$.
We have
$$\dim (X\cap Y)=k-2\;\mbox{ and }\;\dim(X+Y)=k+2,$$ 
which means that the orthogonal complement of $X\cap Y$ in $X+Y$
is $4$-dimensional. We denote this subspace by $M$. 
The subspaces
$$X'=X\cap M\;\mbox{ and }\; Y'=Y\cap M$$
are $2$-dimensional. 
Since $X,Z,Y$ and $X,Z',Y$ are geodesics in $\Gamma^{\perp}_{k}(H)$, 
each of $Z,Z'$ contains $X\cap Y$ and is contained in $X+Y$
which implies that
$$S=Z\cap M\;\mbox{ and }\; S'=Z'\cap M$$
are distinct $2$-dimensional subspaces.
Recall that $Z\cap Z'$ is a $(k-1)$-dimensional subspace of $X$ containing $X\cap Y$,
and consequently,
$$Z\cap Z'\cap M =S\cap S'$$
is a $1$-dimensional subspace of $X'$.

Our next step is to show that $S$ and $S'$ are compatible to both $X',Y'$. 
We will use the following fact \cite[Lemma 1.14]{Pankov-book}: 
if a closed subspace $A\subset H$ is compatible to closed subspaces $B,C\subset H$, 
then $A$ is also compatible to $B\cap C$, $\overline{B+C}$ and $B^{\perp}$.
Since $Z$ is compatible to $X\cap Y$ and $X+Y$, it is compatible to 
$$M=(X+Y)\cap (X\cap Y)^{\perp}.$$
Then $Z$ is compatible to $X'=X\cap M$. So, $X'$ is compatible to $Z$ and $M$ 
(the latter follows from the fact that $X'\subset M$) and, consequently, it is compatible to $S=Z\cap M$.
Similarly, we establish that $X'$ is compatible to $S'$ and $Y'$ is compatible to both $S,S'$.

Using the fact that $S,S'$ are compatible to both $X',Y'$, 
we establish that $X'$ and $Y'$ are orthogonal which immediately  implies that $X$ and $Y$ are compatible.

It was noted above that $S\cap S'$ is a $1$-dimensional subspace of $X'$.
We denote this subspace by $P$ and write $Q$ for the unique $1$-dimensional subspace of $X'$ orthogonal to $P$.
Since $X'$ and $S$ are compatible $2$-dimensional subspaces,  $Q$ is orthogonal to $S$. 
Similarly, it is orthogonal to $S'$.
Then $Q$ is orthogonal to $S+S'$.
Since $$Z=(X\cap Y)+S\;\mbox{ and }\;Z'=(X\cap Y)+S'$$
are ortho-adjacent to $Y=(X\cap Y)+Y'$,
the subspaces $S$ and $S'$ intersect $Y'$ in $1$-dimensional subspaces $P_1$ and $P_2$,
respectively.
These subspaces are distinct (otherwise, they coincide with $S\cap S'$ 
and $X'\cap Y'\ne 0$ which is impossible).
Therefore, $Y'=P_1+P_2$ is contained in $S+S'$,  and consequently, $Q$ is orthogonal to $Y'$.

Now, we show that $P$ is orthogonal to $Y'$. 
Denote by $Q_i$ the $1$-dimensional subspace of $Y'$ orthogonal to $P_i$, $i=1,2$.
We have $Q_1\ne Q_2$, since $P_1\ne P_2$.
Then $Q_1$ is orthogonal to $S$ (since $S$ and $Y'$ are compatible) and, for the same reasons,
$Q_2$ is orthogonal to $S'$. 
This implies that both $Q_1,Q_2$ are orthogonal to $P=S\cap S'$,
and hence, $Y'=Q_1+Q_2$ is orthogonal to $P$.
Then $X'=P+Q$ is orthogonal to $Y'$ and we get the claim.
\end{proof}

By assumption, we have $4\le 2k\le \dim H$ and distinguish the following three cases:
\begin{enumerate}
\item[$\bullet$] $k\le \dim H-4$ which means that $k\ge 4$ or $k=3$, $\dim H\ge 7$ or $k=2$, $\dim H\ge 6$.
\item[$\bullet$] $k=\dim H-3$ which means that $\dim H=2k=6$ or $k=2$, $\dim H=5$.
\item[$\bullet$] $k=\dim H-2$ which means that $\dim H=2k=4$.
\end{enumerate}
We need the following clarification  of Lemma \ref{lemma-adnoncom}.

\begin{lemma}\label{lemma2-3}
For non-compatible  adjacent $k$-dimensional subspaces $X,Y\subset H$
the following assertions are fulfilled:
\begin{enumerate}
\item[{\rm (1)}] If $k\le \dim H-4$, then 
there are infinitely many $k$-dimensional subspaces $Z\subset H$ ortho-adjacent to both $X,Y$ 
and such that there are infinitely many $k$-dimensional subspaces $Z'\subset H$ ortho-adjacent to $X,Y,Z$.
\item[{\rm (2)}] If $k=\dim H -3$, then 
there are infinitely many $k$-dimensional subspaces $Z\subset H$ ortho-adjacent to both $X,Y$ 
and such that there is precisely one $Z'\subset H$ ortho-adjacent to $X,Y,Z$.
\item[{\rm (3)}] If $\dim H=2k=4$, then there are precisely two $2$-dimensional subspaces 
ortho-adjacent to both $X,Y$; these subspaces are orthogonal.
\end{enumerate}
\end{lemma}

\begin{proof}
(1). Since $k\le \dim H-4$ and $X+Y$ is $(k+1)$-dimensional, we have
$$\dim (X+Y)^{\perp}\ge 3$$
and for every $1$-dimensional subspace $P\subset (X+Y)^{\perp}$ there are infinitely many 
$1$-dimensional subspaces $Q\subset (X+Y)^{\perp}$ orthogonal to $P$.
For any such $P$ and $Q$ the $k$-dimensional subspaces
\begin{equation}\label{eq-subs}
P+(X\cap Y), Q+(X\cap Y), X,Y
\end{equation}
are mutually ortho-adjacent.

(2). If $k=\dim H -3$, then the subspace $(X+Y)^{\perp}$ is $2$-dimensional
and for any $1$-dimensional subspace $P\subset (X+Y)^{\perp}$
there is a unique $1$-dimensional subspace $Q\subset (X+Y)^{\perp}$ orthogonal to $P$.
As above, the $k$-dimensional  subspaces \eqref{eq-subs}
are mutually ortho-adjacent. 
The $k$-dimensional subspace $S+W$, 
where $S$ is the orthogonal complement of $X\cap Y$ in $X+Y$ and 
$W$ is a $(k-2)$-dimensional subspace of $X\cap Y$ (see Lemma \ref{lemma-ad-2}),
is not adjacent to $P+(X\cap Y)$.
So, $Q+(X\cap Y)$ is the unique $k$-dimensional subspace ortho-adjacent to $X,Y, P+(X\cap Y)$.

(3). 
Suppose that $\dim H=2k=4$.
By Lemma \ref{lemma-adnoncom}, there are precisely two $2$-dimensional subspaces
ortho-adjacent to both $X,Y$:
one of them is the sum of $X\cap Y$ and $(X+Y)^{\perp}$ and
the second is the orthogonal complement of $X\cap Y$ in $X+Y$.
\end{proof}

We can characterise the adjacency relation in terms of ortho-adjacency if $k\le \dim H-3$.

\begin{lemma}\label{lemma2-4}
Let $X$ and $Y$ be $k$-dimensional subspaces of $H$ such that the distance between $X,Y$
in $\Gamma^{\perp}_{k}(H)$ is equal to $2$. Then the following assertions are fulfilled:
\begin{enumerate}
\item[{\rm (1)}] In the case when $k\le \dim H-4$, the subspaces $X,Y$ are adjacent and non-compatible
if and only if
there are infinitely many $k$-dimensional subspaces $Z\subset H$ ortho-adjacent to both $X,Y$ 
and such that there are infinitely many $k$-dimensional subspaces $Z'\subset H$ ortho-adjacent to $X,Y,Z$.
\item[{\rm (2)}] In the case when $k=\dim H -3$, the subspaces $X,Y$ are adjacent and non-compatible
if and only if
there are infinitely many $k$-dimensional subspaces $Z\subset H$ ortho-adjacent to both $X,Y$ 
and such that there is precisely one $Z'\subset H$ ortho-adjacent to $X,Y,Z$.
\end{enumerate}
\end{lemma}

\begin{proof}
(1). If $X,Y$ are adjacent and non-compatible, then 
the required condition  holds by the statement (1) of Lemma \ref{lemma2-3}.
Conversely, suppose that  
there are infinitely many $k$-dimensional subspaces $Z\subset H$ ortho-adjacent to both $X,Y$ 
and such that there are infinitely many $k$-dimensional subspaces $Z'\subset H$ ortho-adjacent to $X,Y,Z$.
If $X,Y$ are not adjacent, then their intersection is $(k-2)$-dimensional.
Since there is a pair of ortho-adjacent $k$-dimensional subspaces of $H$ which are ortho-adjacent to both $X,Y$,
Lemma \ref{lemma2-2} implies that $X$ and $Y$ are compatible.
Then, by Lemma \ref{lemma2-1}, for every $k$-dimensional subspace $Z\subset H$ ortho-adjacent to both $X,Y$
there are precisely two $k$-dimensional subspaces of $H$ ortho-adjacent to $X,Y,Z$
which contradicts our assumption. 
So, $X,Y$ are adjacent. 
Since the distance between $X$ and $Y$ in $\Gamma^{\perp}_{k}(H)$ is equal to $2$, they are non-compatible.

(2). We use the same arguments and the statement (2) of Lemma \ref{lemma2-3} instead of the statement (1).
\end{proof}

Now, we can prove Theorem \ref{theorem-ortho-gr}.
Let $f$ be an automorphism of $\Gamma^{\perp}_{k}(H)$.
If $k\ge 3$  or $\dim H\ge 5$ and $k=2$,
then $f$ is an  automorphism of $\Gamma_{k}(H)$ by Lemma \ref{lemma2-4} and
the required statement follows from Lemma \ref{lemma-aut}.

\begin{rem}{\rm
In the case when $\dim H=2k=4$,
the adjacency relation cannot be characterised in terms of ortho-adjacency.
If $\dim H=4$ and $X,Y$ are non-compatible adjacent $2$-dimensional subspaces of $H$,
then there are precisely two $2$-dimensional subspaces $Z,Z'\subset H$ ortho-adjacent to both $X,Y$
and $Z$ is orthogonal to $Z'$ (the statement (3) from Lemma \ref{lemma2-3}).
By Example \ref{exmp-4-2}, a $2$-dimensional subspace of $H$ is ortho-adjacent to $Y^{\perp}$
if and only if it is ortho-adjacent to $Y$.
The subspaces $X,Y^{\perp}$ are non-compatible and 
their intersection is zero ($X,Y$ are non-compatible, i.e. $X$ does not contain a $1$-dimensional subspace orthogonal to $Y$).
So, there are precisely two $2$-dimensional subspaces of $H$ ortho-adjacent to both $X,Y^{\perp}$ ($Z$ and $Z'$) and these subspaces are orthogonal.
}\end{rem}

\begin{rem}{\rm
In the case when $\dim H=2k=4$, the orthogonality relation can be characterised in terms of 
ortho-adjacency as follows: $2$-dimensional subspaces $X,Y\subset H$ are orthogonal
if and only if for every $2$-dimensional subspace $Z\subset H$ ortho-adjacent to both $X,Y$
there are precisely two $2$-dimensional subspaces of $H$ ortho-adjacent to $X,Y,Z$
(we use Lemma \ref{lemma2-1} and arguments from the proof of Lemma \ref{lemma2-4}).
So, for $\dim H=4$ every automorphism $f$ of $\Gamma^{\perp}_{2}(H)$ is orthogonality preserving in both directions;
since two $2$-dimensional subspaces of $H$ are compatible if and only if 
they are ortho-adjacent or orthogonal, $f$ is compatible preserving in both directions.
Compatibility preserving transformations of Hilbert Grassmannians are determined 
except the case when $\dim H=2k$ is equal $4$ or $6$ (see \cite[Chapter 5]{Pankov-book} and \cite{Pank-comp}).
}\end{rem}

\section{Generalised ortho-Grassmann graphs associated to conjugacy classes of finite-rank self-adjoint operators}

Recall that two operators $A$ and $B$ on $H$  are {\it unitary conjugate} if 
there is a unitary operator $U$ on $H$ such that $B=UAU^{*}$.
For example, any two rank-$m$ projections are unitary conjugate 
and ${\mathcal P}_{m}(H)$ is the conjugacy class formed by all such projections.
Let ${\mathcal C}$ be  a conjugacy class consisting of finite-rank self-adjoint operators on $H$.
This class  is completely determined by the spectrum $\sigma=\{a_1,\dots,a_k\}$ of operators from ${\mathcal C}$ 
(each eigenvalue $a_i$ is real) and  the set $d=\{n_{1},\dots,n_k\}$, 
where $n_i$ is the dimension of the eigenspaces corresponding to the eigenvalue $a_i$.
The eigenspaces of every operator from ${\mathcal C}$ are mutually orthogonal,
their sum is $H$
and the eigenspace corresponding to every non-zero $a_i$ is finite-dimensional.
If one of $a_i$ is zero, then the kernels of operators from ${\mathcal C}$  are non-trivial.
If $H$ is infinite-dimensional, then $0\in \sigma$ and the kernels are infinite-dimensional
(since operators from ${\mathcal C}$ are of finite rank and their kernels are the orthogonal complements of the images).

In what follows, the conjugacy class ${\mathcal C}$ will be denoted by ${\mathcal G}(\sigma, d)$
and called the $(\sigma, d)$-{\it Grassmannian}.
For example, if $\sigma=\{0,\lambda\}$, where $\lambda$ is a non-zero real number,  
and $d=\{\dim H-m,m\}$, then ${\mathcal G}(\sigma, d)$ is the conjugacy class 
$\lambda{\mathcal P}_{m}(H)$ formed by all $\lambda$-multiples of rank-$m$ projections.
If $\dim H=n$ is finite and $d=\{n\}$,
then ${\mathcal G}(\sigma, d)$ consists of a unique operator which is a scalar multiple of the identity.

If $A\in {\mathcal G}(\sigma, d)$ and 
$X_i$ is the eigenspace of $A$ corresponding to $a_i$,
then
$$A=\sum^{k}_{i=1}a_{i}P_{X_i}.$$
Let $S(d)$ be the group of all permutations $\delta$ on $\{1,\dots,k\}$ such that
$n_{\delta(i)}=n_i$
(the group is trivial if all $n_i$ are mutually distinct).
For every permutation $\delta \in S(d)$  the operator
$$\delta(A)=\sum^{k}_{i=1}a_i P_{X_{\delta(i)}}$$
belongs to ${\mathcal G}(\sigma, d)$.
For example, if $\dim H=2m$, $\sigma=\{0,1\}$ and $d=\{m,m\}$,
then  ${\mathcal G}(\sigma, d)={\mathcal P}_{m}(H)$
and $S(d)$  coincides with $S_2$;
furthermore, if $\delta$ is the non-trivial element of $S(d)$, 
then for every projection $P_{X}\in{\mathcal P}_{m}(H)$ 
we have $\delta(P_{X})=P_{X^{\perp}}$.

Operators $A,B\in {\mathcal G}(\sigma, d)$ are called {\it adjacent} if 
the following conditions are satisfied:
\begin{enumerate}
\item[(A1)] the rank of $A-B$ is equal to $2$,  
\item[(A2)] the kernel and the image of $A-B$ are invariant to both $A,B$;
\end{enumerate}
see  \cite{PPZ} for the details.
This concept admits a simple interpretation in terms of eigenspaces.
If $H$ is infinite-dimensional, then $0\in \sigma$ and the kernels of operators from  ${\mathcal G}(\sigma, d)$ are infinite-dimensional, but their codimension is finite. 
Two closed subspaces of the same finite codimension are called {\it adjacent} if their orthogonal complements are
adjacent, and we say that these subspaces are {\it ortho-adjacent} if the orthogonal complements are
ortho-adjacent.
For every $i\in \{1,\dots,k\}$ denote by $X_i$ and $Y_i$ the eigenspaces of $A$ and $B$ (respectively)
corresponding to the eigenvalue $a_i$.
Then $A$ and $B$ are adjacent  if and only if 
there are distinct $i,j\in \{1,\dots,k\}$ such that the following assertions are fulfilled:
\begin{enumerate}
\item[$\bullet$] $X_{t}=Y_{t}$ for all $t\ne i,j$, and consequently, $X_i+X_j= Y_i +Y_j$;
\item[$\bullet$] $X_i$ and $X_{j}$ are adjacent to $Y_i$ and $Y_j$, respectively\footnote{
Since $X_j$ and $Y_j$ are the orthogonal complements of $X_i$ and $Y_i$ (respectively) in 
$X_i+X_j= Y_i +Y_j$, the subspaces $X_i$ and $Y_i$ are adjacent if and only if $X_j$ and $Y_j$ are adjacent.}.
\end{enumerate}
In this case, the operators $A,B$ are said to be $(i,j)$-{\it adjacent}.
For example,  rank-$m$ projections $P_X$ and $P_Y$ are adjacent if and only if 
the $m$-dimensional subspaces $X$ and $Y$ are adjacent;
furthermore, $P_{X}$ and $P_{Y}$   commute if and only if $X$ and $Y$ are ortho-adjacent.
If $A$ and $B$ commute, then each of these operators commute with $A-B$
which immediately implies (A2). We say that $A,B$ are {\it commutatively adjacent}
if they commute and satisfy (A1). 
This means that $A,B$ are $(i,j)$-adjacent  for some distinct $i,j\in \{1,\dots,k\}$;
furthermore, $X_i$ and $X_j$ are ortho-adjacent to $Y_i$ and $Y_j$, respectively
(it is well-known that two compact self-adjoint operators commute if and only if their eigenspaces are mutually compatible).

The {\it generalised Grassmann graph} $\Gamma(\sigma, d)$ is the simple graph whose vertex set
is ${\mathcal G}(\sigma,d)$ and two operators are connected by an edge if they are adjacent.
If ${\mathcal G}(\sigma, d)={\mathcal P}_{m}(H)$, then $\Gamma(\sigma, d)=\Gamma_{m}(H)$.
The graph $\Gamma(\sigma, d)$ is connected \cite{PPZ}.
Any other $\Gamma(\sigma', d)$ is isomorphic to $\Gamma(\sigma, d)$.
Indeed, if $\sigma'=\{a'_1,\dots,a'_k\}$, then the correspondence 
\begin{equation}\label{iso}
\sum^{k}_{i=1}a_{i}P_{X_i}\to\sum^{k}_{i=1}a'_{i}P_{X_i}
\end{equation}
defines an isomorphism between these graphs.
In particular, if $k=2$ (i.e. there are precisely two eigenvalues for operators from ${\mathcal G}(\sigma,d)$),
then $\Gamma(\sigma, d)$ can be identified with the Grassmann graph $\Gamma_{m}(H)$, where 
$m$ is $n_1$ or $n_2$ (at least one of $n_1,n_2$ is finite).
For every unitary or anti-unitary operator $U$ on $H$ the transformation 
sending every $A\in {\mathcal G}(\sigma, d)$ to $UAU^*$   is an automorphism of $\Gamma(\sigma, d)$
preserving each type of adjacency and
for every non-identity $\delta\in S(d)$ the transformation sending every 
$A\in{\mathcal G}(\sigma, d)$ to $\delta(A)$ is an automorphism of $\Gamma(\sigma,d)$
which does not preserve all types of adjacency.
The following is an analogue of Chows's theorem concerning automorphisms of $\Gamma(\sigma, d)$.

\begin{theorem}[\cite{PPZ}]\label{theorem-PPZ}
Suppose that $k\ge 3$ and $n_i>1$ for all $i\in \{1,\dots,k\}$.
Then for every automorphism $f$ of $\Gamma(\sigma, d)$
there are a unitary or anti-unitary operator $U$ on $H$ and a permutation  $\delta\in S(d)$
such that 
$$f(A)=U\delta(A)U^{*}$$
for all $A\in {\mathcal G}(\sigma,d)$.
In particular, every automorphism of $\Gamma(\sigma, d)$ preserving each type of adjacency
is induced by a unitary or anti-unitary operator on $H$.  
\end{theorem}

\begin{rem}{\rm
The above statement fails for $k=2$ (since $\Gamma(\sigma, d)$ is identified with $\Gamma_{m}(H)$, 
$m\in \{n_1,n_2\}$).
}\end{rem}

Denote by $\Gamma^c(\sigma, d)$ the simple graph whose vertex set
is ${\mathcal G}(\sigma,d)$ and two operators are connected by an edge if they are commutatively adjacent.
Then $\Gamma^c(\sigma, d)=\Gamma^{\perp}_{m}(H)$ if ${\mathcal G}(\sigma, d)={\mathcal P}_{m}(H)$.
Any other $\Gamma^c(\sigma', d)$ is isomorphic to $\Gamma^c(\sigma, d)$
(if $\sigma'=\{a'_1,\dots,a'_k\}$, then \eqref{iso} is a graph isomorphism).
If $k=2$, then $\Gamma^c(\sigma, d)$ is identified with the ortho-Grassmann graph $\Gamma^{\perp}_{m}(H)$, 
where $m$ is $n_1$ or $n_2$.

\begin{prop}
The graph $\Gamma^c(\sigma, d)$ is connected.
\end{prop}

\begin{proof}
Proposition 1 gives the claim for  $k=2$.
Suppose that $k\ge 3$ and $A,B\in {\mathcal G}(\sigma,d)$ are non-commutative and $(i,j)$-adjacent
for some $i,j\in \{1,\dots,k\}$.
For every $t\in \{1,\dots,k\}$ denote by $X_t$ and $Y_t$ the eigenspaces of $A$ and $B$ (respectively)
corresponding to $a_t$.
Then $X_t=Y_t$ for all $t\ne i,j$ and $X_i,X_j$ are adjacent to $Y_i,Y_j$, respectively. 
We write $V$ for the $(n_i+n_j)$-dimensional subspace $X_i+X_j=Y_i+Y_j$
and take any $n_i$-dimensional subspace $Z_i\subset V$ ortho-adjacent to both $X_i,Y_i$.
Let $Z_j$ be the orthogonal complement of $Z_i$ in $V$.
The $n_2$-dimensional subspace $Z_j$ is ortho-adjacent to both $X_j,Y_j$.
Consider the operator $C\in {\mathcal G}(\sigma,d)$ 
whose eigenspace corresponding to $a_{t}$ with $t\ne i,j$ is $X_t=Y_t$
and the eigenspaces corresponding to $a_i$ and $a_j$ are $Z_i$ and $Z_j$, respectively.
This operator is commutatively $(i,j)$-adjacent to both $A,B$.
The connectedness of $\Gamma^c(\sigma, d)$ follows from the fact that
$\Gamma(\sigma, d)$ is connected.
\end{proof}

Theorem \ref{theorem-ortho-gr} can be generalised as follows.

\begin{theorem}\label{theorem-ortho-gr2}
Suppose that $k\ge 2$ and $n_i>1$ for all $i\in \{1,\dots,k\}$.
Also, we require that there is at most one $i$ such that $n_i=2$.
Then for every automorphism $f$ of $\Gamma^c(\sigma, d)$
there are a unitary or anti-unitary operator $U$ on $H$ and a permutation  $\delta\in S(d)$
such that 
$$f(A)=U\delta(A)U^{*}$$
for all $A\in {\mathcal G}(\sigma,d)$.
\end{theorem}  

If $k=2$, then $\Gamma^c(\sigma, d)$ is identified with $\Gamma^{\perp}_{m}(H)$, $m\in \{n_1,n_2\}$
and Theorem \ref{theorem-ortho-gr2} follows from Theorem \ref{theorem-ortho-gr}
(see Remark \ref{rem-theorem2}).
If $k=2$ and $n_1=n_2=2$, then $\dim H=n_1+n_2=4$, the graph $\Gamma^c(\sigma, d)$ is identified with $\Gamma^{\perp}_2(H)$ and the statement fails by Example \ref{exmp-4-2}.

We prove Theorem \ref{theorem-ortho-gr2} for $k\ge 3$ in the next section.
We will use a modification of arguments from Section 4
to investigate geodesics of length $2$ in $\Gamma^c(\sigma, d)$.

\section{Proof of Theorem \ref{theorem-ortho-gr2}}
Suppose that $k\ge 3$, $n_i>1$ for all $i\in \{1,\dots,k\}$
and there is at most one $i$ such that $n_i=2$.

Let $A$ and $B$ be operators from ${\mathcal G}(\sigma,d)$
such that the distance between them in $\Gamma^c(\sigma,d)$ is equal to $2$.
For every $i\in \{1,\dots,k\}$ we denote by $X_i$ and $Y_i$ the eigenspaces of $A$ and $B$ (respectively)
corresponding to the eigenvalue $a_i$.
Consider any $C\in{\mathcal G}(\sigma,d)$ commutatively adjacent to both $A,B$.
There are precisely two eigenspaces of $C$ distinct from the corresponding eigenspaces of $A$
and the same holds for $B$. 
Therefore, there are at most four indices $i$ such that $X_i\ne Y_i$
(we have $X_i\ne Y_i$ for at least two $i$, otherwise $A=B$).

We start from the case when $X_i\ne Y_i$ for precisely two indices $i$.

\begin{lemma}\label{lemma3-1}
Suppose that $X_i\ne Y_i$, $X_j\ne Y_j$ and $X_t=Y_t$ for $t\ne i,j$.
If $C\in{\mathcal G}(\sigma,d)$ is commutatively adjacent to both $A,B$,
then it is commutatively $(i,j)$-adjacent to $A$ and $B$.
\end{lemma}

\begin{proof}
For every $t\in \{1,\dots,k\}$ we denote by $Z_t$ the eigenspace of $C$ corresponding to $a_t$. 
Suppose that $A,C$ are commutatively $(i',j')$-adjacent and $\{i,j\}\ne \{i',j'\}$. 
If $\{i,j\}\cap \{i',j'\}=\emptyset$, then 
$$Z_i=X_i\ne Y_i,\;\;Z_j=X_j\ne Y_j,\;\; Z_{i'}\ne X_{i'}=Y_{i'},\;\;Z_{j'}\ne X_{j'}=Y_{j'},$$
i.e. $Z_t\ne Y_t$ for four indices $t$ which means that $B,C$ are not adjacent.
Therefore, $\{i,j\}\cap \{i',j'\}$ is one-element.
Without loss of generality we assume that $i\ne i'$ and $j=j'$, i.e.
$A,C$ are commutatively $(i',j)$-adjacent.
Then 
$$Z_{i'}\ne X_{i'}=Y_{i'}\;\mbox{ and }\;Z_{i}=X_{i}\ne Y_{i},$$
and, since $B,C$ are commutatively adjacent, they are commutatively $(i,i')$-adjacent.

So, $A,C$ are commutatively $(i',j)$-adjacent and $B,C$ are commutatively $(i,i')$-adjacent.
This implies that $X_i=Z_i$ is ortho-adjacent to $Y_i$ and 
$X_j$ is ortho-adjacent to $Z_j=Y_j$.
Then $A,B$ are commutatively $(i,j)$-adjacent which contradicts our assumption.
\end{proof}

As in Lemma \ref{lemma3-1}, we suppose that $X_i\ne Y_i$, $X_j\ne Y_j$ and $X_t=Y_t$ if $t\ne i,j$.
For every $C\in {\mathcal G}(\sigma,d)$ commutatively adjacent to both $A,B$
the eigenspaces corresponding to $a_i$ and $a_j$ are ortho-adjacent to $X_i,Y_i$ and $X_j,Y_j$
(respectively) and the eigenspace corresponding to $a_t$, $t\ne i,j$ coincides with $X_t=Y_t$.
Denote by $V_{ij}$ the $(n_i+n_j)$-dimensional subspace $X_i+X_j=Y_i+Y_j$.
Recall that at least one of $n_i,n_j$ is finite.
If $n_i$ is finite, then
the distance between $X_i$ and $Y_i$ in $\Gamma^{\perp}_{n_i}(V_{ij})$ is equal to $2$
and one of the following possibilities is realised:
\begin{enumerate}
\item[(1)] $X_i,Y_i$ are adjacent and non-compatible,
\item[(2)] $\dim(X_i\cap Y_i)=n_i -2$.
\end{enumerate} 
Since $X_j,Y_j$ are the orthogonal complements of $X_i,Y_i$ in $V_{ij}$,
the same holds for the index $j$ 
(if $n_j=\infty$, then the possibility (2) is written as follows: the codimension of $X_j\cap Y_j$ in both $X_j,Y_j$
is equal to $2$). 
In the case (1), the operators $A,B$ are non-commutative and adjacent.
Since both $n_i,n_j$ are not less than $2$ and at least one of them is not less than  $3$, we have 
$$\dim V_{ij}=n_i+n_j\ge 5$$ 
and 
$$\min\{n_i,n_j\}\le n_i +n_j-3.$$
Lemma \ref{lemma2-4} implies the following.

\begin{lemma}\label{lemma3-2}
If $X_i\ne Y_i$, $X_j\ne Y_j$ and $X_t=Y_t$ for $t\ne i,j$,
then the following assertions are fulfilled:
\begin{enumerate}
\item[{\rm (1)}] 
In the case when $\min\{n_i,n_j\}\le n_{i}+n_{j}-4$, 
the operators $A,B$ are adjacent and non-commutative if and only if
there are infinitely many $C\in {\mathcal G}(\sigma,d)$ commutatively adjacent to both $A,B$ 
and such that there are infinitely many $C'\in {\mathcal G}(\sigma,d)$ commutatively adjacent to $A,B,C$.
\item[{\rm (2)}] In the case when $\min\{n_i,n_j\}=n_{i}+n_{j}-3$,
the operators $A,B$ are adjacent and non-commutative if and only if
there are infinitely many $C\in {\mathcal G}(\sigma,d)$ commutatively adjacent to both $A,B$ 
and such that there is precisely one $C'\in {\mathcal G}(\sigma,d)$ commutatively adjacent to $A,B,C$.
\end{enumerate}
\end{lemma}

\begin{lemma}\label{lemma3-3}
If $X_i\ne Y_i$ for three or four indices $i$, then the number of operators 
$C\in{\mathcal G}(\sigma,d)$ commutatively adjacent to both $A,B$ is finite.
\end{lemma}

\begin{proof}
Suppose that $X_i\neq Y_i$ precisely for $i\in\{1,2,3,4\}$.
If $C\in{\mathcal G}(\sigma,d)$ is commutatively adjacent to both $A,B$,
then it is $(s,t)$-adjacent to $A$  for some $s,t\in\{1,2,3,4\}$
and $(s',t')$-adjacent to $B$ for $\{s',t'\}=\{1,2,3,4\}\setminus \{s,t\}$
(otherwise there is $i\in \{1,2,3,4\}$ such that $X_i=Y_i$). 
This means that the eigenspace of $C$ corresponding to $a_{j}$ with $j\in \{s',t'\}$ is $X_j$,
the eigenspace of $C$ corresponding to $a_{j}$, $j\in \{s,t\}$ is $Y_j$ 
and the eigenspaces of $C$ associated to the remaining $a_t$ are $X_t=Y_t$.
So, there are at most six operators $C\in{\mathcal G}(\sigma,d)$ commutatively adjacent to both $A,B$.
  
Consider the case when $X_i\neq Y_i$ precisely for $i\in\{1,2,3\}$.
Suppose that $C\in{\mathcal G}(\sigma,d)$ is $(s,t)$-adjacent to $A$ and $(s',t')$-adjacent to $B$.
If $t\not\in \{1,2,3\}$, then the eigenspace of $C$ corresponding to $a_{t}$ is distinct from $X_t=Y_{t}$,
and consequently, $t\in \{s',t'\}$. Without loss of generality we assume that $t=t'$.
Then for $j\in\{1,2,3\}\setminus\{s,s'\}$ we have $X_j=Y_j$ which is impossible. 
So, $s,s',t,t'$ belong to $\{1,2,3\}$. 
This means that $C$ is $(s,j)$-adjacent to $B$ or $(t,j)$-adjacent to $B$
for $\{s,t,j\}=\{1,2,3\}$.

If $C$ is $(s,j)$-adjacent to $B$, then the eigenspace of $C$ corresponding to $a_t$ is $Y_t$.
Since $C$ is $(s,t)$-adjacent to $A$, the eigenspace of $C$ corresponding to $a_j$ is $X_j$.
Then $X_j,Y_t$ are orthogonal and the eigenspace of $C$ corresponding to $a_s$ is
the orthogonal complement of $X_j+Y_t$ in
$$X_1+X_2+X_3=Y_1+Y_2+Y_3.$$
Therefore, there is at most one $C\in{\mathcal G}(\sigma,d)$
which is $(s,t)$-adjacent to $A$ and $(s,j)$-adjacent to $B$.
Similarly, there exists at most one $C\in{\mathcal G}(\sigma,d)$
which is $(s,t)$-adjacent to $A$ and $(t,j)$-adjacent to $B$.

Since there are precisely three distinct $2$-element subsets of $\{1,2,3\}$,
we can get at most six operators $C\in{\mathcal G}(\sigma,d)$
commutatively adjacent to both $A,B$.
\end{proof}

Combining Lemmas \ref{lemma3-2} and \ref{lemma3-3}, we obtain the following characterisation of
adjacency in terms of commutative adjacency. 

\begin{lemma}\label{lemma3-4}
Suppose that the distance between $A,B\in{\mathcal G}(\sigma,d)$ in $\Gamma^c(\sigma,d)$ is equal to $2$.
Then $A,B$ are adjacent and non-commutative if and only if one of the following possibilities is realized:
\begin{enumerate}
\item[{\rm (1)}] 
There are infinitely many $C\in {\mathcal G}(\sigma,d)$ commutatively adjacent to both $A,B$ 
and such that there are infinitely many $C'\in {\mathcal G}(\sigma,d)$ commutatively adjacent to $A,B,C$.
\item[{\rm (2)}] 
There are infinitely many $C\in {\mathcal G}(\sigma,d)$ commutatively adjacent to both $A,B$ 
and such that there is precisely one $C'\in {\mathcal G}(\sigma,d)$ commutatively adjacent to $A,B,C$.
\end{enumerate}
\end{lemma} 

Lemma \ref{lemma3-4} shows that every automorphism of $\Gamma^c(\sigma,d)$ is
an automorphism of $\Gamma(\sigma,d)$ and 
Theorem \ref{theorem-PPZ} gives the claim.

\end{document}